\documentclass[12pt, reqno]{amsart}
\usepackage{amsmath, amsthm, amscd, amsfonts, amssymb, graphicx, color}
\usepackage[bookmarksnumbered, colorlinks, plainpages]{hyperref}

\newtheorem{theorem}{Theorem}[section]
\newtheorem{lemma}[theorem]{Lemma}

\theoremstyle{definition}
\newtheorem{definition}[theorem]{Definition}

\theoremstyle{remark}
\newtheorem{remark}[theorem]{Remark}
\numberwithin{equation}{section}

\begin{document}


\title[Local and 2-local derivations on Lie matrix rings]{Local and 2-local derivations on Lie matrix rings over commutative involutive rings}

\author[Sh.A. Ayupov]{Sh.A. Ayupov$^{1,2}$}

\address{$^{1}$
V.I.Romanovskiy Institute of Mathematics
Uzbekistan Academy of Sciences, Tashkent, Uzbekistan}

\address{$^{2}$
National University of Uzbekistan, Tashkent, Uzbekistan}
\email{\textcolor[rgb]{0.00,0.00,0.84}{shavkat.ayupov@mathinst.uz}}

\author[F.~N. Arzikulov]{F.~N. Arzikulov$^{1,3}$}

\address{$^{3}$ Andizhan State University, Andizhan, Uzbekistan}
\email{\textcolor[rgb]{0.00,0.00,0.84}{arzikulovfn@rambler.ru}}

\author[S.~M. Umrzaqov]{S.~M. Umrzaqov$^3$}

\email{\textcolor[rgb]{0.00,0.00,0.84}{sardor.umrzaqov1986@gmail.com}}



\begin{abstract}
In the present paper we prove that every 2-local inner
derivation on the Lie ring of skew-adjoint matrices over a commutative $*$-ring
is an inner derivation. We also apply our technique
to various Lie algebras of infinite dimensional
skew-adjoint matrix-valued maps on a set, and prove
that every 2-local spatial derivation on such algebras is a spatial derivation.
We also show that every local spatial derivation on the above Lie algebras is a derivation.
\end{abstract}

\maketitle

\section{Introduction}

In the present paper we consider 2-local and local derivations on Lie rings and Lie algebras.
The study of 2-local derivations began in the paper \cite{S} of \v{S}emrl. In \cite{S} \v{S}emrl introduced the notion of 2-local derivations
and described 2-local derivations on the algebra $B(H)$ of all bounded linear operators on the infinite-dimensional
separable Hilbert space $H$. Later a number of papers were devoted to 2-local maps on different types of rings,
algebras, Banach algebras and Banach spaces.

In the present paper 2-local derivations on the Lie ring of skew-adjoint matrices over a commutative involutive ring are described.
If in an associative algebra $\mathcal{A}$ we take the Lie multiplication $[a,b]=ab-ba$, then we obtain
the Lie algebra $(\mathcal{A},[,])$. In this case, every 2-local inner derivation of the algebra $\mathcal{A}$ is
a derivation if and only if every 2-local inner derivation of the Lie algebra $(\mathcal{A},[,])$ is a derivation.
In general, for any Lie algebra $(\mathcal{L},[,])$, the Lie multiplication $[,]$ is generated by an
associative multiplication. Therefore in number of papers the proofs of results on 2-local Lie derivations
are based on associative multiplication. For example, in \cite{CLW} of L. Chen, F. Lu and T. Wang 2-local Lie derivations of operator algebras
over Banach spaces are described. The first paper, where 2-local derivations of Lie algebras which are not generated by associative algebras,
are described, is the paper \cite{AyuKudRak16} of Sh. Ayupov, K. Kudaybergenov and I. Rakhimov.
They proved that every 2-local derivation on a finite-dimensional
semi-simple Lie algebra $\mathcal{L}$ over an algebraically closed field of characteristic zero is a derivation.
They also showed that each finite-dimensional nilpotent Lie  algebra $\mathcal{L}$ with $\dim \mathcal{L}\geq 2$
admits a 2-local derivation which is not a derivation. At the same time, in \cite{LaiZheng15}
X. Lai and Z.X. Chen give a description of 2-local Lie derivations for the case of finite dimensional
simple Lie algebras. L. Liu characterized 2-local Lie derivations on a semi-finite
factor von Neumann algebra in \cite{Liu16}. Recently, in \cite{HLAH}  it is proved that
every 2-local Lie derivation on factor von Neumann algebras, UHF algebras and the Jiang-Su algebra is a Lie derivation.

In \cite{AA5} the first and second authors of the present paper investigated 2-local inner derivations
on the Lie ring $K_n(\Re)$ of skew-symmetric $n\times n$ matrices over a commutative
associative ring $\Re$ and proved that every 2-local inner derivation is a derivation.

In the present paper we study inner derivations and 2-local inner
derivations on Lie rings of skew-adjoint matrices over a commutative involutive ring.
 After Section 2 - Preliminaries, In Section 3 we prove that each 2-local inner derivation on the Lie
ring $K_n(\Re)$ of skew-adjoint $n\times n$ matrices over a commutative unital involutive ring $\Re$ is
a derivation.
As a corollary we establish that every 2-local inner derivation on the Lie
algebra $K_n(\mathcal{A})$ of skew-adjoint $n\times n$ matrices over a commutative unital involutive algebra $\mathcal{A}$
is a derivation.

We also study 2-local spatial derivations on various Lie algebras of
infinite dimensional skew-adjoint matrix-valued maps on a set.
We prove that every 2-local spatial derivation on subalgebras of the algebra
$M(\Omega,B_{sk}(H))$ of all maps from $\Omega$ to $B_{sk}(H)$ for an arbitrary set $\Omega$ and the algebra $B(H)$ of
all bounded linear operators on a separable Hilbert space $H$, under some conditions, is a derivation.
The problems considered here are firstly mentioned in \cite{AA4} (Problem 1).

Sections 4 and 5 are devoted to the description of local inner derivations on the above considered algebras.
For this propose we apply a similar technique to various Lie algebras of
skew-adjoint infinite dimensional matrix-valued maps on a set and prove
that every local spatial derivation on such algebras is a spatial derivation.
It should be noted that a number of results concerning local derivations are obtained, partially, in
\cite{AK}, \cite{CR}, \cite{HL}, \cite{HLAH}, \cite{JB}, \cite{KR}, \cite{LS}, \cite{PY}.

\section{Preliminaries}

Let $\Re$ be a ring (an algebra) (nonassociative in general). Recall that an additive (respectively linear) map $D : \Re\to \Re$ is called
a derivation, if $D(xy)=D(x)y+xD(y)$ for
any two elements $x$, $y\in \Re$.

A map $\Delta : \Re\to \Re$  (neither  additive nor linear, in general) is called a 2-local derivation, if for
any two elements $x$, $y\in \Re$ there exists a derivation
$D_{x,y}:\Re\to \Re$ such that $\Delta (x)=D_{x,y}(x)$, $\Delta
(y)=D_{x,y}(y)$.

Let $\Re$ be an associative ring (algebra). A derivation $D$ on $\Re$
is called an inner derivation, if there exists an element $a\in
\Re$ such that
$$
D(x)=ax-xa, x\in \Re.
$$
A map $\Delta : \Re\to \Re$ (neither  additive nor linear, in general) is called a 2-local inner derivation,
if for any two elements $x$, $y\in \Re$ there exists an element
$a\in \Re$ such that $\Delta (x)=ax-xa$, $\Delta (y)=ay-ya$.

Let $\Re$ be a Lie ring (Lie algebra). Given an element  $a$ in  $\Re$,
the map $R_a(x)=[a,x]$, $x\in \Re$ is a Lie derivation. Such derivation
is called an inner derivation of $\Re$. A map $\Delta$ is called a 2-local inner derivation, if for each pair of elements
$x$, $y\in \Re$ there is an inner derivation $R_a$ of $\Re$ such that $\Delta(x)=R_a(x)$, $\Delta(y)=R_a(y)$.

Let $\mathcal{A}$ be an associative unital ring (algebra over an arbitrary field).
The vector space $\mathcal{A}$ with respect to the Lie multiplication
$$
[a,b]=ab-ba, a, b\in \mathcal{A}
$$
is a Lie ring (respectively Lie algebra). This Lie ring (respectively Lie algebra)  will be  denoted by $(\mathcal{A}, [ , ])$. For $a\in \mathcal{A}$
the map
$$
R_a(x)=[a,x], x\in \mathcal{A}.
$$
 is a derivation on $\mathcal{A}$  both as an associative ring and as a Lie ring.
Every inner  derivation of the Lie ring $(\mathcal{A},[,])$ is an inner derivation of the associative ring
$\mathcal{A}$. And also every inner derivation $R_a=ax-xa$,
$x\in \mathcal{A}$ is an inner derivation of Lie ring $(\mathcal{A},[,])$.
 Thus it is clear that every 2-local inner derivation of the Lie ring (Lie algebra) $(\mathcal{A},[,])$
is a 2-local inner derivation of the associative ring (respectively associative algebra) $\mathcal{A}$. And vice-versa, every 2-local inner derivation of the
associative ring (associative algebra) $\mathcal{A}$ is a 2-local inner derivation of the Lie ring (respectively Lie algebra) $(\mathcal{A},[,])$.

Let, now, $\mathcal{A}$ be a $*$-ring ($*$-algebra) and $\mathcal{A}_k$ be the set of all skew-adjoint elements of
$\mathcal{A}$. Then $(\mathcal{A}_k,[,])$ is a Lie ring (respectively Lie algebra). We take
$a\in\mathcal{A}_k$ and the inner derivation
$$
R_a(x)=[a,x], x\in \mathcal{A}_k.
$$
Then $R_a$  may be considered as an inner derivation on $\mathcal{A}$.
Therefore every inner derivation of the Lie ring (Lie algebra) $(\mathcal{A}_k,[,])$ is extended to
an inner derivation of the $*$-ring (respectively $*$-algebra) $\mathcal{A}$.

Concerning  2-local inner derivations, it is not possible, in general,
to extend a 2-local inner derivations from the Lie
ring $(\mathcal{A}_k,[,])$ to a 2-local inner derivation on the involutive ring $\mathcal{A}$.
Therefore we have to give a straightforward  proofs of the main results below.

\section{2-Local derivations on the Lie ring of skew-adjoint matrices over a commutative $*$-ring}

Let $\Re$ be a unital associative ring, $M_n(\Re)$ be the
matrix ring over $\Re$, $n>1$.
Let $\{e_{i,j}\}_{i,j=1}^n$ be the set of matrix units in
$M_n(\Re)$, i.e. $e_{i,j}$ is a matrix with components
$a^{i,j}={\bf 1}$ and $a^{k,l}={\bf 0}$ if $(i,j)\neq(k,l)$, where
${\bf 1}$ is the identity element, ${\bf 0}$ is the zero element of
$\Re$, and a matrix $a\in M_n(\Re)$ is written as $a=\sum_{k,l=1}^n
a^{k,l}e_{k,l}$, where $a^{k,l}\in \Re$ for $k,l=1,2,\dots, n$.

\medskip

Let $\Re$ be a commutative unital involutive ring, $M_n(\Re)$ be the associative ring of $n\times n$ matrices
over $\Re$, $n>1$. In this case the vector space
$$
K_n(\Re)=\{(a^{i,j})_{i,j=1}^n\in M_n(\Re):
(a^{i,j})^*=-a^{j,i}, i,j=1,2,\dots ,n\}
$$
is a Lie ring with respect to the Lie multiplication
$$
[a,b]=ab-ba, a, b\in K_n(\Re).
$$
Throughout this section, let $s_{i,j}=e_{i,j}-e_{j,i}$ for every pair of different indices $i$, $j$
in $\{1,2,\dots ,n\}$. Further we suppose that the ring $\Re$ contains an imaginary unit $I$,
i.e. $I^2=-e$, where $e$ is an identity element of $\Re$.

\begin{lemma}  \label{3.4}
Let $\Delta$ be a 2-local derivation on $K_n(\Re)$ and fix
arbitrary pairwise different indices $i$, $j$, $p$.

1) let $R_a$, $R_b$ be the derivations on $K_n(\Re)$, generated by elements
$a$, $b\in H_n(\Re)$ such that
$$
\Delta (s_{i,j})=R_a(s_{i,j})=R_b(s_{i,j}).
$$
Then the following equalities are valid
$$
a^{i,j}+a^{j,i}=b^{i,j}+b^{j,i},
$$
$$
a^{i,i}-a^{j,j}=b^{i,i}-b^{j,j}.
$$

2) Let $R_a$, $R_b$ be the derivations on $K_n(\Re)$, generated by elements
$a$, $b\in K_n(\Re)$ such that
$$
\Delta(s_{i,j})=R_a(s_{i,j}), \Delta(s_{i,p})=R_b(s_{i,p}).
$$
Then the following equality is valid
$$
a^{i,j}+a^{j,i}=b^{i,j}+b^{j,i}.
$$
\end{lemma}

\begin{proof} 1):
From $R_a(s_{i,j})=R_b(s_{i,j})$ it follows that
\[
as_{i,j}-s_{i,j}a=bs_{i,j}-s_{i,j}b.
\]
Hence,
\[
e_{j,j}as_{i,j}e_{j,j}-e_{j,j}s_{i,j}ae_{j,j}=e_{j,j}bs_{i,j}e_{j,j}-e_{j,j}s_{i,j}be_{j,j}
\]
and
\[
e_{j,j}ae_{i,j}e_{j,j}+e_{j,j}e_{j,i}ae_{j,j}=e_{j,j}be_{i,j}e_{j,j}+e_{j,j}e_{j,i}be_{j,j},
\]
\[
a^{j,i}e_{j,j}+a^{i,j}e_{j,j}=b^{j,i}e_{j,j}+b^{i,j}e_{j,j}.
\]
Also,
\[
e_{i,i}as_{i,j}e_{j,j}-e_{i,i}s_{i,j}ae_{j,j}=e_{i,i}bs_{i,j}e_{j,j}-e_{i,i}s_{i,j}be_{j,j}
\]
and
\[
e_{i,i}ae_{i,j}e_{j,j}-e_{i,i}e_{i,j}ae_{j,j}=e_{i,i}be_{i,j}e_{j,j}-e_{i,i}e_{i,j}be_{j,j},
\]
\[
a^{i,i}e_{i,j}-a^{j,j}e_{i,j}=b^{i,i}e_{i,j}-b^{j,j}e_{i,j}.
\]
This ends the proof of 1).

2): There exist
$x$, $y\in K_n(\Re)$ such that
$$
\Delta(s_{i,j})=R_x(s_{i,j}), \Delta(s_{i,p})=R_x(s_{i,p}).
$$
We have
$$
as_{i,j}-s_{i,j}a=xs_{i,j}-s_{i,j}x,
$$
$$
bs_{i,p}-s_{i,p}b=xs_{i,p}-s_{i,p}x.
$$
By the assertion 1) of this lemma
$$
a^{i,j}+a^{j,i}=x^{i,j}+x^{j,i}.
$$
Also we have
$$
e_{p,p}(bs_{i,p}-s_{i,p}b)e_{j,j}=e_{p,p}(xs_{i,p}-s_{i,p}x)e_{j,j},
$$
$$
e_{p,i}be_{j,j}=e_{p,i}xe_{j,j}, b^{i,j}e_{p,j}=x^{i,j}e_{p,j},
$$
i.e. $b^{i,j}=x^{i,j}$. Similarly we have $b^{j,i}=x^{j,i}$.
Hence
$$
a^{i,j}+a^{j,i}=b^{i,j}+b^{j,i}.
$$
\end{proof}

\begin{lemma} \label{3.41}
Let $\Delta$ be a 2-local derivation on $K_n(\Re)$ and, for
arbitrary pairwise different indices $i$, $j$, $p$,
let $R_a$, $R_b$ be the derivations on $K_n(\Re)$, generated by elements
$a$, $b\in K_n(\Re)$ such that
$$
\Delta (s_{i,p})=R_a(s_{i,p}), \Delta (s_{p,j})=R_b(s_{p,j}).
$$
Then the following equalities hold
$$
e_{i,i}ae_{j,j}=e_{i,i}be_{j,j},
e_{j,j}ae_{i,i}=e_{j,j}be_{i,i}.
$$
\end{lemma}

\begin{proof}
There exist $x\in K_n(\Re)$ such that
$$
\Delta (s_{i,p})=R_x(s_{i,p}), \Delta (s_{p,j})=R_x(s_{p,j}).
$$
Then
$$
as_{i,p}-s_{i,p}a
=xs_{i,p}-s_{i,p}x,
$$
$$
bs_{p,j}-s_{p,j}b
=xs_{p,j}-s_{p,j}x,
$$
$$
e_{j,j}ae_{i,p}=e_{j,j}xe_{i,p},
e_{p,i}ae_{j,j}=e_{p,i}xe_{j,j},
$$
$$
e_{i,i}be_{j,p}=e_{i,i}xe_{j,p},
e_{p,j}be_{i,i}=e_{p,j}xe_{i,i}.
$$
Hence
$$
e_{j,j}ae_{i,i}=e_{j,j}xe_{i,i},
e_{i,i}ae_{j,j}=e_{i,i}xe_{j,j},
$$
$$
e_{i,i}be_{j,j}=e_{i,i}xe_{j,j},
e_{j,j}be_{i,i}=e_{j,j}xe_{i,i}.
$$
and
$$
e_{j,j}ae_{i,i}=e_{j,j}be_{i,i},
e_{i,i}ae_{j,j}=e_{i,i}be_{j,j}.
$$
\end{proof}

Let $\Delta$ be a 2-local derivation on $K_n(\Re)$ and, for arbitrary pairwise different indices $i$, $j$, $p$,
let $R_a$ be the derivation on $K_n(\Re)$, generated by an element $a\in K_n(\Re)$ such that
$$
\Delta (s_{i,p})=R_a(s_{i,p}), \Delta (s_{p,j})=R_a(s_{p,j}).
$$
By lemma \ref{3.41} the following elements are well-defined
$$
a_{i,j}=e_{i,i}ae_{j,j}=a^{i,j}e_{i,j}, a^{i,j}\in \Re,
$$
$$
a_{j,i}=e_{j,j}ae_{i,i}=a^{j,i}e_{j,i}, a^{j,i}\in \Re,
$$
$$
a=\sum_{i,j=1, i\neq j}^n a_{i,j}.
$$
In these notations the following lemma is valid.

\begin{lemma}  \label{2.5}
Let $\Re$ be a commutative unital involutive ring. Let $\Delta$ be a 2-local inner derivation
on $K_n(\Re)$. For arbitrary but fixed different indices $i$, $j$ let
$$
\Delta (s_{i,j})=R_d(s_{i,j})
$$
for an appropriate element $d\in K_n(\Re)$. Then
$$
\Delta(s_{i,j})=
as_{i,j}-s_{i,j}a+(d^{i,i}-d^{j,j})e_{i,j}+(d^{j,j}-d^{i,i})e_{j,i}.
$$
\end{lemma}

\begin{proof}
By 2) of lemma \ref{3.4} we have
$$
\Delta(s_{i,j})=ds_{i,j}-s_{i,j}d=
$$
$$
(1-e_{i,i}-e_{j,j})ds_{i,j}-s_{i,j}d(1-e_{i,i}-e_{j,j})
$$
$$
+(e_{i,i}+e_{j,j})ds_{i,j}-s_{i,j}d(e_{i,i}+e_{j,j})=
$$
$$
(1-e_{i,i}-e_{j,j})de_{i,j}-e_{i,j}d(1-e_{i,i}-e_{j,j})
$$
$$
-(1-e_{i,i}-e_{j,j})de_{j,i}+e_{j,i}d(1-e_{i,i}-e_{j,j})
$$
$$
+(e_{i,i}+e_{j,j})ds_{i,j}-s_{i,j}d(e_{i,i}+e_{j,j})=
$$
$$
(1-e_{i,i}-e_{j,j})ae_{i,j}-e_{i,j}a(1-e_{i,i}-e_{j,j})
$$
$$
-(1-e_{i,i}-e_{j,j})ae_{j,i}+e_{j,i}a(1-e_{i,i}-e_{j,j})+
$$
$$
(e_{i,i}+e_{j,j})ds_{i,j}-s_{i,j}d(e_{i,i}+e_{j,j}).
$$
At the same time
$$
(e_{i,i}+e_{j,j})ds_{i,j}-s_{i,j}d(e_{i,i}+e_{j,j})=
$$
$$
e_{i,i}de_{i,j}-e_{i,i}de_{j,i}+e_{j,j}de_{i,j}-e_{j,j}de_{j,i}-
$$
$$
e_{i,j}de_{i,i}+e_{j,i}de_{i,i}-e_{i,j}de_{j,j}+e_{j,i}de_{j,j}=
$$
$$
(d^{i,i}-d^{j,j})e_{i,j}+(-d^{i,j}-d^{j,i})e_{i,i}+
$$
$$
(d^{j,i}+d^{i,j})e_{j,j}+(d^{i,i}-d^{j,j})e_{j,i}=
$$
$$
(d^{i,i}-d^{j,j})e_{i,j}+(-a^{i,j}-a^{j,i})e_{i,i}+
$$
$$
(a^{j,i}+a^{i,j})e_{j,j}+(d^{i,i}-d^{j,j})e_{j,i}=
$$
$$
(d^{i,i}-d^{j,j})e_{i,j}-a_{i,j}e_{j,i}-e_{i,j}a_{j,i}+
$$
$$
a_{j,i}e_{i,j}+e_{j,i}a_{i,j}+(d^{i,i}-d^{j,j})e_{j,i}.
$$
by 1) of lemma \ref{3.4}. Hence
$$
\Delta(s_{i,j})=ae_{i,j}-e_{i,j}a-ae_{j,i}+e_{j,i}a+
$$
$$
(d^{i,i}-d^{j,j})e_{i,j}+(d^{j,j}-d^{i,i})e_{j,i}=
$$
$$
as_{i,j}-s_{i,j}a+(d^{i,i}-d^{j,j})e_{i,j}+(d^{j,j}-d^{i,i})e_{j,i}.
$$
This ends the proof.
\end{proof}

Let $\Re$ be a unital involutive ring,
and consider the element $x_o=\sum_{k=1}^{n-1}s_{k,k+1}\in K_n(\Re)$.  Fix different indices $i_o$, $j_o$. Let
$c\in K_n(\Re)$ be an element such that
$$
\Delta(s_{i_o,j_o})=[c,s_{i_o,j_o}], \Delta(x_o)=[c,x_o].
$$

Put $c=\sum_{i,j=1}^n c^{i,j}e_{i,j}\in K_n(\Re)$ and
$\bar{a}=\sum_{i,j=1, i\neq j}^n a_{i,j}+\sum_{i=1}^n a_{i,i}$, where
$a_{i,i}=c^{i,i}e_{i,i}$, $i=1,2,\dots,n$.

\begin{lemma}  \label{3.6}
Let $\Re$ be a unital involutive ring and, for arbitrary different
indices $k$, $l$, let $b\in K_n(\Re)$ be an element
such that
$$
\Delta(s_{k,l})=[b,s_{k,l}], \Delta(x_o)=[b,x_o].
$$
Then $c^{k,k}-c^{l,l}=b^{k,k}-b^{l,l}$.
\end{lemma}

\begin{proof}
We may assume that  $k<l$. We have
$$
\Delta(x_o)=[c,x_o]=[b,x_o].
$$
Hence
$$
cx_o-x_oc=bx_o-x_ob,
$$
$$
e_{k,k}(cx_o-x_oc)e_{k+1,k+1}=e_{k,k}(bx_o-x_ob)e_{k+1,k+1}
$$
and
$$
c^{k,k}-c^{k+1,k+1}=b^{k,k}-b^{k+1,k+1}.
$$
Then for the sequence
$$
(k,k+1),(k+1,k+2)\dots (l-1,l)
$$
we have
$$
c^{k,k}-c^{k+1,k+1}=b^{k,k}-b^{k+1,k+1},
c^{k+1,k+1}-c^{k+2,k+2}=b^{k+1,k+1}-b^{k+2,k+2},\dots
$$
$$
c^{l-1,l-1}-c^{l,l}=b^{l-1,l-1}-b^{l,l}.
$$
Hence
$$
c^{k,k}-b^{k,k}=c^{k+1,k+1}-b^{k+1,k+1},
c^{k+1,k+1}-b^{k+1,k+1}=c^{k+2,k+2}-b^{k+2,k+2},\dots
$$
$$
c^{l-1,l-1}-b^{l-1,l-1}=c^{l,l}-b^{l,l}.
$$
Therefore $c^{k,k}-b^{k,k}=c^{l,l}-b^{l,l}$, i.e.
$c^{k,k}-c^{l,l}=b^{k,k}-b^{l,l}$. The proof is complete.
\end{proof}

\begin{theorem} \label{2.6}
Let $\Re$ be a commutative unital involutive ring, and let $K_n(\Re)$
be the Lie ring of skew-adjoint $n\times n$ matrices over $\Re$. Then
any 2-local inner derivation on the matrix Lie ring $K_n(\Re)$ is an inner derivation.
\end{theorem}

\begin{proof}
We prove that each inner 2-local derivation $\Delta$ on $K_n(\Re)$ satisfies
the condition
$$
\Delta (x)=R_{\bar{a}}(x)=\bar{a}x-x\bar{a}, x\in K_n(\Re),
$$
where the element $\bar{a}$ is defined before Lemma \ref{2.5}.
Let $x$ be an arbitrary element in $K_n(\Re)$ and let
$d(ij)\in K_n(\Re)$ be an element such that
$$
\Delta (s_{i,j})=R_{d(ij)}(s_{i,j}), \Delta (x)=R_{d(ij)}(x)
$$
and $i\neq j$. Then
$$
\Delta (s_{i,j})=D_{d(ij)}(s_{i,j}), \Delta (x)=D_{d(ij)}(x),
$$
i.e.,
$$
\Delta(s_{i,j})=d(ij)(s_{i,j})-(s_{i,j})d(ij) \,\, \text{and}\,\,
\Delta(x)=d(ij)x-xd(ij).
$$
By lemma \ref{2.5} we have the following equalities
$$
\Delta(s_{i,j})=d(ij)(s_{i,j})-(s_{i,j})d(ij)=
$$
$$
(e_{i,i}+e_{j,j})d(ij)(s_{i,j})-(s_{i,j})d(ij)(e_{i,i}+e_{j,j})+
$$
$$
(1-(e_{i,i}+e_{j,j}))d(ij)(s_{i,j})-(s_{i,j})d(ij)(1-(e_{i,i}+e_{j,j}))=
$$
$$
as_{i,j}-s_{i,j}a+
(d(ij)^{i,i}-d(ij)^{j,j})e_{i,j}+(d(ij)^{j,j}-d(ij)^{i,i})e_{j,i}.
$$
for all different indices $i$, $j$.
From this equality we get
$$
(e_{i,i}+e_{j,j})d(ij)(s_{i,j})-(s_{i,j})d(ij)(e_{i,i}+e_{j,j})=
$$
$$
(a_{i,j}+a_{j,i})s_{i,j}-s_{i,j}(a_{i,j}+a_{j,i})
$$
$$
+(d(ij)^{i,i}-d(ij)^{j,j})e_{i,j}+(d(ij)^{j,j}-d(ij)^{i,i})e_{j,i}.
$$
Hence,
$$
(1-(e_{i,i}+e_{j,j}))d(ij)e_{i,i}=(1-(e_{i,i}+e_{j,j}))ae_{i,i},
$$
$$
e_{i,i}d(ij)(1-(e_{i,i}+e_{j,j}))=e_{i,i}a(1-(e_{i,i}+e_{j,j})),
$$
$$
e_{j,j}d(ij)(1-(e_{i,i}+e_{j,j}))=e_{j,j}a(1-(e_{i,i}+e_{j,j})),
$$
$$
(1-(e_{i,i}+e_{j,j}))d(ij)e_{j,j}=(1-(e_{i,i}+e_{j,j}))ae_{j,j}
$$
for all different $i$ and $j$. At the same time
$$
e_{j,j}d(ij)e_{i,j}+e_{j,i}d(ij)e_{j,j}=(d(ij)^{j,i}+d(ij)^{i,j})e_{j,j}
$$
$$
=(a^{j,i}+a^{i,j})e_{j,j}=a_{j,i}e_{i,j}+e_{j,i}a_{i,j},
$$
$$
e_{i,i}d(ij)e_{j,i}+e_{i,j}d(ij)e_{i,i}=(d(ij)^{i,j}+d(ij)^{j,i})e_{i,i}
$$
$$
=(a^{i,j}+a^{j,i})e_{i,i}=a_{i,j}e_{j,i}+e_{i,j}a_{j,i}
$$
by lemma \ref{3.4}, and
$$
e_{j,j}d(ij)e_{i,j}+e_{j,i}d(ij)e_{j,j}=e_{j,j}ae_{i,j}+e_{j,i}ae_{j,j},
$$
$$
e_{i,i}d(ij)e_{j,i}+e_{i,j}d(ij)e_{i,i}=e_{i,i}ae_{j,i}+e_{i,j}ae_{i,i}.
$$
Therefore
$$
(1-e_{i,i})d(ij)e_{i,i}=(1-e_{i,i})ae_{i,i},
$$
$$
e_{j,j}d(ij)(1-e_{j,j})=e_{j,j}a(1-e_{j,j}).
$$
Similarly we have
$$
e_{i,i}d(ij)(1-e_{i,i})=e_{i,i}a(1-e_{i,i}),
$$
$$
(1-e_{j,j})d(ij)e_{j,j}=(1-e_{j,j})ae_{j,j}.
$$
Let $v\in K_n(\Re)$ be elements such that
$$
\Delta(s_{i,j})=vs_{i,j}-s_{i,j}v
$$
and
$$
\Delta(x_o)=vx_o-x_ov.
$$

Let $v=\sum_{k,q=1}^n v^{k,q}e_{k,q}$. Then $v^{i,i}-v^{j,j}=c^{i,i}-c^{j,j}$ by lemma \ref{3.6}. We have
$v^{i,i}-v^{j,j}=d(ij)^{i,i}-d(ij)^{j,j}$ since
$$
vs_{i,j}-s_{i,j}v=d(ij)s_{i,j}-s_{i,j}d(ij).
$$
Hence
$$
c^{i,i}-c^{j,j}=d(ij)^{i,i}-d(ij)^{j,j},
c^{j,j}-c^{i,i}=d(ij)^{j,j}-d(ij)^{i,i}.
$$

Therefore we have
$$
e_{j,j}\Delta(x)e_{i,i}=e_{j,j}(d(ij)x-xd(ij))e_{i,i}=
$$
$$
e_{j,j}d(ij)(1-e_{j,j})xe_{i,i}+
e_{j,j}d(ij)e_{j,j}xe_{i,i}-e_{j,j}x(1-e_{i,i})d(ij)e_{i,i}-e_{j,j}xe_{i,i}d(ij)e_{i,i}=
$$
$$
e_{j,j}\sum_{\xi,\eta=1, \xi\neq\eta}^n a^{\xi,\eta}e_{\xi,\eta}xe_{i,i}
-e_{j,j}x\sum_{\xi,\eta=1, \xi\neq\eta}^n a^{\xi,\eta}e_{\xi,\eta}e_{i,i}+ e_{j,j}d(ij)e_{j,j}xe_{i,i}-e_{j,j}xe_{i,i}d(ij)e_{i,i}=
$$
$$
e_{j,j}\sum_{\xi,\eta=1,\xi\neq\eta}^n a^{\xi,\eta}e_{\xi,\eta}xe_{i,i}
-e_{j,j}x\sum_{\xi,\eta=1,\xi\neq\eta}^n a^{\xi,\eta}e_{\xi,\eta}e_{i,i}
+c^{j,j}e_{j,j}xe_{i,i}-e_{j,j}xe_{i,i}c^{i,i}e_{i,i}=
$$
$$
e_{j,j}\sum_{\xi,\eta=1,\xi\neq\eta}^n a^{\xi,\eta}e_{\xi,\eta}xe_{i,i}
-e_{j,j}x\sum_{\xi,\eta=1,\xi\neq\eta}^n a^{\xi,\eta}e_{\xi,\eta}e_{i,i}+
$$
$$
e_{j,j}(\sum_{\xi=1}^n a^{\xi,\xi}e_{\xi,\xi})xe_{i,i}-e_{j,j}x(\sum_{\xi=1}^n a^{\xi,\xi}e_{\xi,\xi})e_{i,i}=
$$
$$
e_{j,j}\sum_{\xi,\eta=1}^n a^{\xi,\eta}e_{\xi,\eta}xe_{i,i}-e_{j,j}x\sum_{\xi,\eta=1}^n a^{\xi,\eta}e_{\xi,\eta}e_{i,i}=
e_{j,j}(\bar{a}x-x\bar{a})e_{i,i}.
$$

Let $d(ii)$, $v$, $w\in K_n(\Re)$ be elements such that
$$
\Delta(Ie_{i,i})=d(ii)Ie_{i,i}-Ie_{i,i}d(ii) \,\, \text{and}\,\,
\Delta(x)=d(ii)x-xd(ii),
$$
$$
\Delta(Ie_{i,i})=vIe_{i,i}-Ie_{i,i}v,
\Delta(s_{i,j})=vs_{i,j}-s_{i,j}v,
$$
and
$$
\Delta(Ie_{i,i})=wIe_{i,i}-Ie_{i,i}w,
\Delta(s_{i,j})=w(s_{i,j})-(s_{i,j})w.
$$
Then
$$
(1-e_{i,i})ae_{i,i}=(1-e_{i,i})ve_{i,i}=(1-e_{i,i})d(ii)e_{i,i},
$$
and
$$
e_{i,i}a(1-e_{i,i})=e_{i,i}w(1-e_{i,i})=e_{i,i}d(ii)(1-e_{i,i}).
$$
By the properties of the Peirce components we have
$$
(1-e_{i,i})ae_{i,i}=\sum_{\xi,\eta=1,\xi\neq\eta}^n a_{\xi,\eta}e_{i,i},
$$
$$
e_{i,i}a(1-e_{i,i})=e_{i,i}\sum_{\xi,\eta=1,\xi\neq\eta}^n a_{\xi\eta}.
$$

Hence
$$
e_{i,i}\Delta(x)e_{i,i}=e_{i,i}(d(ii)x-xd(ii))e_{i,i}=
$$
$$
e_{i,i}d(ii)(1-e_{i,i})xe_{i,i}+
e_{i,i}d(ii)e_{i,i}xe_{i,i}-e_{i,i}x(1-e_{i,i})d(ii)e_{i,i}-e_{i,i}xe_{i,i}d(ii)e_{i,i}=
$$
$$
e_{i,i}a(1-e_{i,i})xe_{i,i}+
e_{i,i}d(ii)e_{i,i}xe_{i,i}-e_{i,i}x(1-e_{i,i})ae_{i,i}-e_{i,i}xe_{i,i}d(ii)e_{i,i}=
$$
$$
e_{i,i}\sum_{\xi,\eta=1,\xi\neq\eta}^n a^{\xi,\eta}e_{\xi,\eta}xe_{i,i}
-e_{i,i}x\sum_{\xi,\eta=1,\xi\neq\eta}^n a^{\xi,\eta}e_{\xi,\eta}e_{i,i}
+e_{i,i}d(ii)e_{i,i}xe_{i,i}-e_{i,i}xe_{i,i}d(ii)e_{i,i}=
$$
$$
e_{i,i}\sum_{\xi,\eta=1,\xi\neq\eta}^n a^{\xi,\eta}e_{\xi,\eta}xe_{i,i}
-e_{i,i}x\sum_{\xi,\eta=1,\xi\neq\eta}^n a^{\xi,\eta}e_{\xi,\eta}e_{i,i}
+c^{i,i}e_{i,i}xe_{i,i}-e_{i,i}xc^{i,i}e_{i,i}=
$$
$$
e_{i,i}\sum_{\xi,\eta=1,\xi\neq\eta}^n a^{\xi,\eta}e_{\xi,\eta}xe_{i,i}
-e_{i,i}x\sum_{\xi,\eta=1,\xi\neq\eta}^n a^{\xi,\eta}e_{\xi,\eta}e_{i,i}+
$$
$$
e_{i,i}(\sum_{\xi=1}^n a^{\xi,\xi}e_{\xi,\xi})xe_{i,i}
-e_{i,i}x(\sum_{\xi=1}^n a^{\xi,\xi}e_{\xi,\xi})e_{i,i}=
$$
$$
e_{i,i}\sum_{\xi,\eta=1}^n a^{\xi,\eta}e_{\xi,\eta}xe_{i,i}
-e_{i,i}x\sum_{\xi,\eta=1}^n a^{\xi,\eta}e_{\xi,\eta}e_{i,i}=e_{i,i}(\bar{a}x-x\bar{a})e_{i,i}.
$$

Hence,
$$
\Delta(x)=\bar{a}x-x\bar{a}
$$
for each $x\in K_n(\Re)$. Therefore
$\Delta$ is an inner derivation on $K_n(\Re)$.
The proof is complete.
\end{proof}

The proof of Theorem \ref{2.6} is also valid for Lie algebras of skew-adjoint matrices over
a commutative involutive algebra. Therefore we have the following statement.

\begin{theorem} \label{2.7}
Let $\mathcal{A}$ be a commutative unital involutive algebra, and let $K_n(\mathcal{A})$, $n>1$,
be the Lie algebra of all skew-adjoint $n\times n$ matrices over $\mathcal{A}$. Then
any 2-local inner derivation on the matrix Lie algebra $K_n(\mathcal{A})$ is an inner derivation.
\end{theorem}

\section{2-local derivations on Lie algebras of skew-adjoint matrix-valued
maps}

Throughout the rest part of the paper, let $\omega_o$ be the countable cardinal
number. Now let
$H$ be a separable complex Hilbert space of dimension $\omega_o$ and, let $B(H)$ be the von Neumann algebra of
all bounded linear operators on $H$. Let $\{e_i\}_{i=1}^\infty$ be a maximal family of orthogonal minimal
projections in $B(H)$ and, let $\{e_{i,j}\}_{i,j=1}^\infty$ be the family of matrix units defined
by $\{e_i\}_{i=1}^\infty$, i.e. $e_{i,i}=e_i$, $e_{i,i}=e_{i,j}e_{j,i}$ and $e_{j,j}=e_{j,i}e_{i,j}$ for
each pair $i$, $j$ of natural numbers.

Let $B_{sk}(H)$ be the vector space of all skew-adjoint matrices in $B(H)$, i.e.
$$
B_{sk}(H)=\{a\in B(H): a^*=-a\}.
$$
Then with respect to Lie multiplication
$$
[a,b]=ab-ba, a,b\in B_{sk}(H)
$$
$B_{sk}(H)$ is a Lie algebra.

Let $\Omega$ be an arbitrary set, $F(\Omega,B_{sk}(H))$ be the
Lie algebra of all maps from $\Omega$ to $B_{sk}(H)$ with the Lie multiplication
$$
[a,b](t)=[a(t),b(t)], t\in\Omega, a, b\in F(\Omega,B_{sk}(H)).
$$
Put
$$
\hat{e}_{i,j}={\bf 1}e_{i,j},
$$
where ${\bf 1}$ is the identity element of the algebra $F(\Omega)$ of all complex-valued maps on $\Omega$.
Throughout the rest sections, put $s_{i,j}=\hat{e}_{i,j}-\hat{e}_{j,i}$ for every pair of different natural numbers $i$, $j$,
and, let $I$ be the imaginary unit.

{\it Definition.} Let $A$ be a Lie algebra and, let $B$ be a Lie subalgebra of $A$.
A derivation $D$ on $B$ is said to be spatial, if $D$ is implemented by an element in $A$,
i.e.
$$
D(x)=[a,x], x\in B,
$$
for some $a\in A$. A 2-local derivation $\Delta$ on $B$ is called 2-local spatial derivation
implemented by elements from $A$, if
for every two elements $x$, $y\in B$ there exists an element
$a\in A$ such that $\Delta(x)=[a,x]$, $\Delta(y)=[a,y]$.

Let, throughout the present section, $\Omega$ be an arbitrary set, $F(\Omega,B_{sk}(H))$ be the
Lie algebra of all maps from $\Omega$ to $B_{sk}(H)$.
Let $(\lambda_n)$ be a sequence of nonzero numbers from ${\mathbb C}$ such that $\sum_{n=1}^\infty \lambda_n\lambda_n^*<\infty$ and
$x_o=\sum_{n=1}^\infty\lambda_n s_{n,n+1}\in F(\Omega,B_{sk}(H))$. Let
$\mathcal{L}$ be a Lie subalgebra of $F(\Omega,B_{sk}(H))$ containing the element $x_o$ and the family
$\{I\hat{e}_{i,i}\}_{i=1}^\infty\cup\{s_{i,j}\}_{i,j=1}^\infty$. The following lemmas hold.

\begin{lemma}  \label{5.2}
Let $\Delta$ be a 2-local spatial derivation
on $\mathcal{L}$ implemented by elements from $F(\Omega,B_{sk}(H))$
and $i$, $j$, $p$ be arbitrary pairwise different natural numbers,

1) let $R_a$, $R_b$ be the spatial derivations on $\mathcal{L}$,
implemented by $a$, $b\in F(\Omega,B_{sk}(H))$, respectively, such that
$$
\Delta (s_{i,j})=R_a(s_{i,j})=R_b(s_{i,j}).
$$
Then the following equalities are valid
$$
a^{i,j}+a^{j,i}=b^{i,j}+b^{j,i},
$$
$$
a^{i,i}-a^{j,j}=b^{i,i}-b^{j,j}.
$$

2) let $R_a$, $R_b$ be the spatial derivations on $\mathcal{L}$,
implemented by $a$, $b\in F(\Omega,B_{sk}(H))$, respectively, such that
$$
\Delta(s_{i,j})=R_a(s_{i,j}), \Delta(s_{i,p})=R_b(s_{i,p}).
$$
Then the following equality is valid
$$
a^{i,j}+a^{j,i}=b^{i,j}+b^{j,i}.
$$
\end{lemma}

\begin{proof}
The proof of this lemma is similar to the proof of lemma \ref{3.4}.
\end{proof}

\begin{lemma} \label{5.41}
Let $\Delta$ be a 2-local spatial derivation
on $\mathcal{L}$ implemented by elements from $F(\Omega,B_{sk}(H))$,
$i$, $j$, $p$ be arbitrary pairwise different natural numbers and,
let $R_a$, $R_b$ be the derivations on $\mathcal{L}$, generated by elements
$a$, $b\in F(\Omega,B_{sk}(H))$ such that
$$
\Delta (s_{i,p})=R_a(s_{i,p}), \Delta (s_{p,j})=R_b(s_{p,j}).
$$
Then the following equalities hold
$$
\hat{e}_{i,i}a\hat{e}_{j,j}=\hat{e}_{i,i}b\hat{e}_{j,j},
\hat{e}_{j,j}a\hat{e}_{i,i}=\hat{e}_{j,j}b\hat{e}_{i,i}.
$$
\end{lemma}

\begin{proof}
The proof of this lemma is also similar to the proof of lemma \ref{3.41}.
\end{proof}

Let $\Delta$ be a 2-local spatial derivation
on $\mathcal{L}$, implemented by elements from $F(\Omega,B_{sk}(H))$,
$i$, $j$, $p$ be arbitrary pairwise different natural numbers, and,
let $R_a$ be the derivation on $\mathcal{L}$, generated by an element $a\in F(\Omega,B_{sk}(H))$ such that
$$
\Delta (s_{i,p})=R_a(s_{i,p}), \Delta (s_{p,j})=R_a(s_{p,j}).
$$
By lemma \ref{5.41} the following elements are well-defined:
$$
a_{i,j}=\hat{e}_{i,i}a\hat{e}_{j,j}=a^{i,j}\hat{e}_{i,j}, a^{i,j}\in \Re,
$$
$$
a_{j,i}=\hat{e}_{j,j}a\hat{e}_{i,i}=a^{j,i}\hat{e}_{j,i}, a^{j,i}\in \Re,
$$
$$
a=\sum_{i,j=1, i\neq j}^\infty a_{i,j}.
$$
In these notations the following lemma is valid.

\begin{lemma}  \label{5.5}
Let $\Delta$ be a 2-local spatial derivation
on $\mathcal{L}$ implemented by elements from $F(\Omega,B_{sk}(H))$, and let
$i$, $j$ be arbitrary different natural numbers and
$$
\Delta (s_{i,j})=R_d(s_{i,j})
$$
for some element $d\in K_n(\Re)$. Then
$$
\Delta(s_{i,j})=
as_{i,j}-s_{i,j}a+(d^{i,i}-d^{j,j})\hat{e}_{i,j}+(d^{j,j}-d^{i,i})\hat{e}_{j,i}.
$$
\end{lemma}

\begin{proof}
The proof of this lemma is also similar to the proof of lemma \ref{2.5}.
\end{proof}

Fix different natural numbers $i_o$, $j_o$. Let
$c\in F(\Omega,B_{sk}(H))$ be an element such that
$$
\Delta(s_{i_oj_o})=[c,s_{i_oj_o}], \Delta(x_o)=[c,x_o].
$$
Put $c=\sum_{i,j=1}^\infty c^{i,j}\hat{e}_{i,j}\in F(\Omega,B_{sk}(H))$ and
$\bar{a}=\sum_{i,j=1,i\neq j}^\infty a_{i,j}+\sum_{i=1}^\infty a_{i,i}$, where
$a_{i,i}=c^{i,i}\hat{e}_{i,i}$ for any $i$. Then $\bar{a}\in F(\Omega,B_{sk}(H))$.
In the above notations we have

\medskip

\begin{lemma} \label{4.51}
Let $k$, $l$ be arbitrary different
natural numbers, and let $d$ be an element in $F(\Omega,B_{sk}(H))$ such that
$$
\Delta(s_{k,l})=[d,s_{k,l}], \Delta(x_o)=[d,x_o].
$$
Then $c^{k,k}-c^{l,l}=d^{k,k}-d^{l,l}$.
\end{lemma}

\begin{proof}
We may assume that  $k<l$. We have
$$
\Delta(x_o)=cx_o-x_oc=dx_o-x_od.
$$
Hence
$$
e_{k,k}(cx_o-x_oc)e_{k+1,k+1}=e_{k,k}(dx_o-x_od)e_{k+1,k+1},
$$
$$
e_{k,k}(c\lambda_k s_{k,k+1}-\lambda_k s_{k,k+1}c)e_{k+1,k+1}=e_{k,k}(d\lambda_ks_{k,k+1}-\lambda_ks_{k,k+1}d)e_{k+1,k+1},
$$
$$
\lambda_k e_{k,k}c\hat{e}_{k,k+1}-\lambda_k \hat{e}_{k,k+1}ce_{k+1,k+1}=\lambda_ke_{k,k}d\hat{e}_{k,k+1}-\lambda_k\hat{e}_{k,k+1}de_{k+1,k+1},
$$
and
$$
c^{k,k}-c^{k+1,k+1}=d^{k,k}-d^{k+1,k+1}.
$$
Thus, this lemma is proved similarly to the proof of
Lemma \ref{3.6}.
\end{proof}

Now we prove the key theorem of this section.

\begin{theorem} \label{5.4}
Let $\Omega$ be an arbitrary set and let $F(\Omega,B_{sk}(H))$ be the
Lie algebra of all maps from $\Omega$ to $B_{sk}(H)$. Let
$\mathcal{L}$ be a Lie subalgebra of $F(\Omega,B_{sk}(H))$ containing the element $x_o$ and the family
$\{I\hat{e}_{i,i}\}_{i=1}^\infty\cup\{s_{i,j}\}_{i,j=1,i\neq j}^\infty$. Then any 2-local spatial derivation
on $\mathcal{L}$ implemented by elements from $F(\Omega,B_{sk}(H))$
is a spatial derivation.
\end{theorem}

\begin{proof}
We prove that each inner 2-local derivation $\Delta$ on $\mathcal{L}$ satisfies
the condition
$$
\Delta (x)=R_{\bar{a}}(x)=\bar{a}x-x\bar{a}, x\in \mathcal{L}
$$
for the element $\bar{a}\in F(\Omega,B_{sk}(H))$ defined above.
Let $i$, $j$ be arbitrary distinct natural numbers. The proofs of the equalities
$$
\hat{e}_{j,j}\Delta(x)\hat{e}_{i,i}=\hat{e}_{j,j}(\bar{a}x-x\bar{a})\hat{e}_{i,i},
$$
$$
\hat{e}_{i,i}\Delta(x)\hat{e}_{i,i}=\hat{e}_{i,i}(\bar{a}x-x\bar{a})\hat{e}_{i,i}
$$
are the same as the proofs of the appropriate equalities in Theorem \ref{2.6}.
If, for arbitrary elements $y$, $z$ in $F(\Omega,B_{sk}(H))$,
$$
\hat{e}_{\xi,\xi}y\hat{e}_{\eta,\eta}=\hat{e}_{\xi,\xi}z\hat{e}_{\eta,\eta}, \xi,\eta=1,2,3,\dots,
$$
then $y=z$. Hence,
$$
\Delta(x)=\bar{a}x-x\bar{a}\in F(\Omega,B_{sk}(H)).
$$
But, $\Delta(x)\in \mathcal{L}$, since $x\in \mathcal{L}$. Therefore
$$
\Delta(x)=\bar{a}x-x\bar{a}\in \mathcal{L}
$$
for any element $x\in \mathcal{L}$. So, $\Delta$ is a spatial derivation
on $\mathcal{L}$. This ends the proof.
\end{proof}

In particular, as a corollary of Theorem \ref{5.4} we have the following theorem.

\begin{theorem} \label{5.5}
Each 2-local inner derivation on the Lie algebra $F(\Omega,B_{sk}(H))$ is an inner derivation.
\end{theorem}

Let $\Omega$ be a hyperstonean compact, $C(\Omega)$
denotes the algebra of all ${\mathbb C}$-valued continuous maps on $\Omega$.
There exists a subalgebra $\mathcal{N}$ in $F(\Omega,B(H))$, containing the family $\{\hat{e}_{i,j}\}_{i,j=1}^\infty$,
which is a von Neumann algebra
with the center isomorphic to $C(\Omega)$ (see \cite[Page 12]{AFN3}). More precisely
$\mathcal{N}$ is a von Neumann algebra of type I.

The vector space
$$
\mathcal{K}=\{a\in \mathcal{N}: a^*=-a\}
$$
of all skew-adjoint elements in $\mathcal{N}$
is a Lie algebra with respect to Lie multiplication
$$
[a,b]=ab-ba, a,b\in \mathcal{K}.
$$
The Lie algebra $\mathcal{K}$ is a Lie subalgebra of $F(\Omega,B_{sk}(H))$
containing the element $x_o$ and the family $\{I\hat{e}_{i,i}\}_{i\in =1}^\infty\cup\{s_{i,j}\}_{i,j=1, i\neq j}^\infty$. So, by Theorem \ref{3.4} we have the following theorem.

\begin{theorem} \label{5.6}
Every 2-local spatial derivation $\Delta$
on $\mathcal{K}$ implemented by elements from in $F(\Omega,B_{sk}(H))$
is a spatial derivation, i.e., there exists an element $a\in F(\Omega,B_{sk}(H))$ such that
$$
\Delta(x)=R_a(x), x\in \mathcal{K}.
$$
\end{theorem}

\begin{remark}  \label{5.611}
Note that it is not necessary that for each pair of elements $x\in\mathcal{K}$ and $y\in F(\Omega,B_{sk}(H))$ the elements $yx$, $xy$, $R_x(y)$ belong to $\mathcal{K}$.
But, by theorem \ref{5.4}, in theorem \ref{5.6}, for every $x\in\mathcal{K}$, the element $R_a(x)$ belongs to $\mathcal{K}$.
\end{remark}

Let $\mathcal{K}(H)$ be the C$^*$-algebra of all compact
operators on the Hilbert space $H$ over ${\mathbb C}$, $\mathcal{K}_{K}(H)$ be
the Lie algebra of all skew-adjoint compact operators on the Hilbert space $H$.
Let $\Omega$ be a topological space. Then
the vector space $C(\Omega,\mathcal{K}_{K}(H))$ of all continuous maps
from $\Omega$ to $\mathcal{K}_{K}(H)$ is a Lie subalgebra of $F(\Omega,B_{sk}(H))$
containing the element $x_o$ and the family $\{I\hat{e}_{i,i}\}_{i=1}^\infty\cup\{s_{i,j}\}_{i,j=1,i\neq j}^\infty$. Therefore, by Theorem \ref{5.4} we have
the following result.

\begin{theorem} \label{5.7}
Every 2-local spatial derivation
on the Lie algebra $C(\Omega,\mathcal{K}_{K}(H))$ implemented by elements from $F(\Omega,B_{sk}(H))$
is a spatial derivation, i.e., there exists an element $a\in F(\Omega,B_{sk}(H))$ such that
$$
\Delta(x)=R_a(x), x\in \mathcal{K}.
$$
\end{theorem}

\begin{remark}
Note that, as in Remark \ref{5.611}, it is not necessary that for each pair of elements $x\in C(\Omega,\mathcal{K}_{K}(H))$ and $y\in F(\Omega,B_{sk}(H))$ the elements $yx$, $xy$, $R_x(y)$ belong to $C(\Omega,\mathcal{K}_{K}(H))$.
But, by theorem \ref{5.4}, in theorem \ref{5.7}, for every $x\in C(\Omega,\mathcal{K}_{K}(H))$, the element $R_a(x)$ belongs to $C(\Omega,\mathcal{K}_{K}(H))$.
\end{remark}

\section{Local derivations on the Lie algebra of skew-adjoint operators on a complex Hilbert space}

Let $\nabla$ be a local inner derivation on $B_{sk}(H)$.
Let $i$, $j$ be pairwise distinct natural numbers and $e=e_{i,i}+e_{j,j}$.
We take the following map
$$
\nabla_{i,j}(x)=e\nabla(x)e, x\in eB_{sk}(H)e.
$$
Then $\nabla_{i,j}$ is a local derivation on the subalgebra $eB_{sk}(H)e$, and,
by \cite{AK}, it is an inner derivation. Hence, there exists
an element $a\in eB_{sk}(H)e$ such that
$$
\nabla_{i,j}(x)=ax-xa, x\in eB_{sk}(H)e.
$$
Put
$$
a_{i,j}=e_{i,i}ae_{j,j},  i,j=1,2,3,\dots.
$$
The set $\{a_{i,j}\}_{i,j=1}^\infty$ is well defined. Indeed, let $e_1=\sum_{k=1}^me_{i_k,j_k}$ for some
subset $\{e_{i_k,j_l}\}_{k,l=1,2,...,m}$ of elements in $\{e_{\xi.\eta}\}_{\xi,\eta=1}^\infty$, including the elements $e_{i,i}$, $e_{j,j}$. Then,
if $b$ is an element in $e_1B_{sk}(H)e_1$ such that
$$
e_1\nabla(x)e_1=bx-xb, x\in e_1B_{sk}(H)e_1,
$$
then $a=ebe$, i.e.,
$$
a_{i,i}=e_{i,i}be_{i,i}, a_{i,j}=e_{i,i}be_{j,j}, a_{j,i}=e_{j,j}be_{i,i}, a_{j,j}=e_{j,j}be_{j,j}.
$$

\begin{lemma} \label{4.0}
Let $\nabla$ be a local inner derivation
on $B_{sk}(H)$. Then for every pair of different natural numbers $i$, $j$
\begin{enumerate}
  \item $\bar{a}=\sum_{k=1}^\infty[(a_{i,k}+a_{j,k})+(a_{k,i}+a_{k,j})]\in B_{sk}(H)$,
  \item
$$
\nabla(Ie_{i,i})=R_{\bar{a}}(Ie_{i,i}), \nabla(s_{i,j})=R_{\bar{a}}(s_{i,j}),
$$
$$
\nabla(I\bar{e}_{i,j})=R_{\bar{a}}(I\bar{e}_{i,j}), \nabla(Ie_{j,j})=R_{\bar{a}}(Ie_{j,j}),
$$
where $I$ is the imaginary unit and $\bar{e}_{i,j}=e_{i,j}+e_{j,i}$.
\end{enumerate}
\end{lemma}

\begin{proof}
Let $d_{i,i}$, $d_{i,j}$, $\bar{d}_{i,j}$ $d_{j,j}$ be elements in $B_{sk}(H)$ such that
$$
\nabla(Ie_{i,i})=d_{i,i}Ie_{i,i}-Ie_{i,i}d_{i,i},
\nabla(s_{i,j})=d_{i,j}s_{i,j}-s_{i,j}d_{i,j},
$$
$$
\nabla(I\bar{e}_{i,j})=\bar{d}_{i,j}I\bar{e}_{i,j}-I\bar{e}_{i,j}\bar{d}_{i,j},
\nabla(Ie_{j,j})=d_{j,j}Ie_{j,j}-Ie_{j,j}d_{j,j}.
$$

Then, since
\[
e_{k,k}(d_{i,i}e_{i,i}-e_{i,i}d_{i,i})e_{l,l}=a_{k,i}e_{i,i}e_{l,l}-e_{k,k}e_{i,i}a_{i,l},
\]
\[
e_{k,k}(d_{i,j}s_{i,j}-s_{i,j}d_{i,j})e_{l,l}=(a_{k,i}+a_{k,j})s_{i,j}e_{l,l}-e_{k,k}s_{i,j}(a_{i,l}+a_{j,l}),
\]
\[
e_{k,k}(\bar{d}_{i,j}I\bar{e}_{i,j}-I\bar{e}_{i,j}\bar{d}_{i,j})e_{l,l}=(a_{k,i}+a_{k,j})\bar{e}_{i,j}e_{l,l}-e_{k,k}\bar{e}_{i,j}(a_{i,l}+a_{j,l}),
\]
\[
e_{k,k}(d_{j,j}e_{j,j}-e_{j,j}d_{j,j})e_{l,l}=a_{k,j}e_{j,j}e_{l,l}-e_{k,k}e_{j,j}a_{j,l}
\]
for every pair $k$, $l$ of natural numbers, we have
\[
e_{k,k}d_{i,i}e_{i,i}=a_{k,i}, e_{i,i}d_{i,i}e_{l,l}=a_{i,l},
\]
\[
e_{k,k}d_{j,j}e_{j,j}=a_{k,j}, e_{j,j}d_{j,j}e_{l,l}=a_{j,l},
\]
i.e., $\bar{a}\in B_{sk}(H)$,
and
\[
e_{k,k}(d_{i,i}e_{i,i}-e_{i,i}d_{i,i})e_{l,l}=e_{k,k}(\bar{a}e_{i,i}-e_{i,i}\bar{a})e_{l,l},
\]
\[
e_{k,k}(d_{i,j}s_{i,j}-s_{i,j}d_{i,j})e_{l,l}=e_{k,k}(\bar{a}s_{i,j}-s_{i,j}\bar{a})e_{l,l},
\]
\[
e_{k,k}(\bar{d}_{i,j}I\bar{e}_{i,j}-I\bar{e}_{i,j}\bar{d}_{i,j})e_{l,l}=e_{k,k}(\bar{a}I\bar{e}_{i,j}-I\bar{e}_{i,j}\bar{a})e_{l,l},
\]
\[
e_{k,k}(d_{j,j}e_{j,j}-e_{j,j}d_{j,j})e_{l,l}=e_{k,k}(\bar{a}e_{j,j}-e_{j,j}\bar{a})e_{l,l}
\]
for every pair $k$, $l$ of natural numbers. Hence,
\[
\nabla(Ie_{i,i})=d_{i,i}Ie_{i,i}-Ie_{i,i}d_{i,i}=\bar{a}Ie_{i,i}-Ie_{i,i}\bar{a},
\]
\[
\nabla(s_{i,j})=d_{i,j}s_{i,j}-s_{i,j}d_{i,j}=\bar{a}s_{i,j}-s_{i,j}\bar{a},
\]
\[
\nabla(I\bar{e}_{i,j})=\bar{d}_{i,j}I\bar{e}_{i,j}-I\bar{e}_{i,j}\bar{d}_{i,j}=\bar{a}I\bar{e}_{i,j}-I\bar{e}_{i,j}\bar{a},
\]
\[
\nabla(Ie_{j,j})=d_{j,j}Ie_{j,j}-Ie_{j,j}d_{j,j}=\bar{a}Ie_{j,j}-Ie_{j,j}\bar{a}.
\]
\end{proof}

\begin{theorem}  \label{4.4}
Every ultraweakly continuous local inner derivation
on $B_{sk}(H)$ is an inner derivation.
\end{theorem}

\begin{proof}
Let $\nabla :B_{sk}(H)\to B_{sk}(H)$ be a local inner derivation.
Then there exist elements $a_{i,i}\in B_{sk}(H)$, $i=1,2,3,\dots$ such that
$$
\nabla(Ie_{i,i})=a_{i,i}Ie_{i,i}-Ie_{i,i}a_{i,i}
$$
for the elements $e_{i,i}$, $i=1,2,3,\dots$. Let $a_1$ be an element in $B_{sk}(H)$
such that
$$
\nabla(I\bar{e}_{i,k})=a_1(I\bar{e}_{i,k})-(I\bar{e}_{i,k})a_1
$$
for $i$, $k=1,2,3,\dots$. Then, by additivity of $\nabla$, we have
$$
a_1(I\bar{e}_{i,k})-(I\bar{e}_{i,k})a_1
=a_{i,i}Ie_{i,i}-Ie_{i,i}a_{i,i}+a_{k,k}Ie_{k,k}-Ie_{k,k}a_{k,k}.
$$
From this equality it follows that
$$
a_1^{i,k}-a_1^{i,k}=-a_{i,i}^{i,k}+a_{k,k}^{i,k}.
$$
Similarly
$$
a_1^{k,i}-a_1^{k,i}=a_{i,i}^{k,i}-a_{k,k}^{k,i}.
$$
Hence
$$
a_{i,i}^{i,k}=a_{k,k}^{i,k},   a_{i,i}^{k,i}=a_{k,k}^{k,i}.                   \eqno{(5.1)}
$$

Let $x_o=\sum_{k=1}^{\infty}(e_{k,{k+1}}-e_{{k+1},k})\in B_{sk}(H)$. Then $x_o\in B_{sk}(H)$ and
there exists an element $a_2$ in $B_{sk}(H)$ such that
$$
\nabla(x_o)=a_2x_o-x_oa_2.               \eqno{(5.2)}
$$
We construct an element $d$ as follows: its diagonal components are $d^{i,i}=a_2^{i,i}$, $i=1,2,3,\dots$ and its nondiagonal components are
$d^{i,j}=a_{i,i}^{i,j}$ when $i, j=1,2,3,\dots$, $i\neq j$. By equalities (5.1)
the element $d$ is constructed correctly. It is clear that
$$
\nabla(Ie_{i,i})=a_{i,i}Ie_{i,i}-Ie_{i,i}a_{i,i}=dIe_{i,i}-Ie_{i,i}d     \eqno{(5.3)}
$$
for $i=1,2,3,\dots$.

Now we prove that
$$
\nabla(s_{i,k})=d(s_{i,k})-(s_{i,k})d,
$$
$$
\nabla(I\bar{e}_{i,k})=d(I\bar{e}_{i,k})-(I\bar{e}_{i,k})d.
$$

For this propose it is sufficient to prove that
$$
\nabla(s_{i,k})=a_{i,k}(s_{i,k})-(s_{i,k})a_{i,k}
=d(s_{i,k})-(s_{i,k})d,
$$
$$
\nabla(I\bar{e}_{i,k})=a_{i,k}(I\bar{e}_{i,k})-(I\bar{e}_{i,k})a_{i,k}
=d(I\bar{e}_{i,k})-(I\bar{e}_{i,k})d.                    \eqno{(5.4)}
$$

To prove the equalities (5.4) it is sufficient to show
$$
a_{i,k}^{k,j}=d^{k,j}, j\in \{1,\dots,n\}\setminus\{k\}                               \eqno{(5.5)}
$$
$$
a_{i,k}^{j,i}=d^{j,i}, j\in \{1,\dots,n\}\setminus\{i\}                               \eqno{(5.6)}
$$
$$
a_{i,k}^{i,i}-a_{i,k}^{k,k}=d^{i,i}-d^{k,k}.                     \eqno{(5.7)}
$$
By additivity of $\nabla$ and lemma \ref{4.0} we have
$$
a_3(Ie_{k,k}+(e_{k,i}-e_{i,k}))-(Ie_{k,k}+(e_{k,i}-e_{i,k}))a_3
$$
$$
=dIe_{k,k}-Ie_{k,k}d+a_{i,k}(e_{k,i}-e_{i,k})-(e_{k,i}-e_{i,k})a_{i,k},
$$
$$
a_3(Ie_{k,k}+I\bar{e}_{i,k})-(Ie_{k,k}+I\bar{e}_{i,k})a_3
$$
$$
=dIe_{k,k}-Ie_{k,k}d+a_{i,k}I\bar{e}_{i,k}-I\bar{e}_{i,k}a_{i,k}.     \eqno{(5.8)}
$$
$$
a_3(Ie_{i,i}+(s_{i,k}))-(Ie_{i,i}+(s_{i,k}))a_3
$$
$$
=dIe_{i,i}-Ie_{i,i}d+a_{i,k}(s_{i,k})-(s_{i,k})a_{i,k},
$$
$$
a_3(Ie_{i,i}+I\bar{e}_{i,k})-(Ie_{i,i}+I\bar{e}_{i,k})a_3
$$
$$
=dIe_{i,i}-Ie_{i,i}d+a_{i,k}I\bar{e}_{i,k}-I\bar{e}_{i,k}a_{i,k}.     \eqno{(5.9)}
$$
Multiplying by $e_{j,j}$, $j\neq i$, $j\neq k$, the last equalities from the left side we have
$$
Ia_3^{j,k}e_{j,k}+a_3^{j,k}e_{j,i}-a_3^{j,i}e_{j,k}=Id^{j,k}e_{j,k}+a_{i,k}^{j,k}e_{j,i}-a_{i,k}^{j,i}e_{j,k},
$$
$$
Ia_3^{j,k}e_{j,k}+Ia_3^{j,i}e_{j,k}+Ia_3^{j,k}e_{j,i}=Id^{j,k}e_{j,k}+Ia_{i,k}^{j,i}e_{j,k}+Ia_{i,k}^{j,k}e_{j,i},
$$
$$
Ia_3^{j,i}e_{j,i}+a_3^{j,i}e_{j,k}-a_3^{j,k}e_{j,i}=Id^{j,i}e_{j,i}+a_{i,k}^{j,i}e_{j,k}-a_{i,k}^{j,k}e_{j,i},
$$
$$
Ia_3^{j,i}e_{j,i}+Ia_3^{j,k}e_{j,i}+Ia_3^{j,i}e_{j,k}=Id^{j,i}e_{j,i}+Ia_{i,k}^{j,k}e_{j,i}+Ia_{i,k}^{j,i}e_{j,k}.
$$
Hence,
$$
a_3^{j,k}=a_{i,k}^{j,k}, a_3^{j,k}-a_3^{j,i}=d^{j,k}-a_{i,k}^{j,i},
$$
$$
a_3^{j,k}=a_{i,k}^{j,k}, a_3^{j,k}+a_3^{j,i}=d^{j,k}+a_{i,k}^{j,i},
$$
$$
a_3^{j,i}=a_{i,k}^{j,i}, a_3^{j,i}-a_3^{j,k}=d^{j,i}-a_{i,k}^{j,k},
$$
$$
a_3^{j,i}=a_{i,k}^{j,i}, a_3^{j,i}+a_3^{j,k}=d^{j,i}+a_{i,k}^{j,k},
$$
i.e.,
$$
a_3^{j,k}=a_{i,k}^{j,k}, a_3^{j,k}=d^{j,k},
$$
$$
a_3^{j,i}=a_{i,k}^{j,i}, a_3^{j,i}=d^{j,i}.
$$
So,
$$
a_{i,k}^{j,k}=d^{j,k}, a_{i,k}^{j,i}=d^{j,i} j\neq i, j\neq k.
$$
Similarly, multiplying by $e_{j,j}$, $j\neq i$, $j\neq k$ the equalities (5.8), (5.9) from the right side we have
$$
a_{i,k}^{k,j}=d^{k,j}, a_{i,k}^{i,j}=d^{i,j} j\neq i, j\neq k.
$$
Hence equalities (5.5), (5.6) are valid.

Similarly, multiplying by $e_{i,i}$ from the both right and left sides the equalities (5.8) we have
$$
e_{i,i}a_3e_{k,i}+e_{i,k}a_3e_{i,i}
=e_{i,i}a_{i,k}e_{k,i}+e_{i,k}a_{i,k}e_{i,i},
$$
$$
e_{i,i}a_3e_{k,i}-Ie_{i,k}a_3e_{i,i}
=e_{i,i}a_{i,k}Ie_{k,i}-Ie_{i,k}a_{i,k}e_{i,i}.
$$
Hence,
$$
a_3^{i,k}+a_3^{k,i}=a_{i,k}^{i,k}+a_{i,k}^{k,i},
$$
$$
Ia_3^{i,k}-Ia_3^{k,i}=a_{i,k}^{i,k}I-Ia_{i,k}^{k,i}
$$
and
$$
a_3^{i,k}=a_{i,k}^{i,k},  a_3^{k,i}=a_{i,k}^{k,i}.          \eqno{(5.10)}
$$

Multiplying the equalities (4.7) by $e_{i,i}$ from the left side and by $e_{k,k}$ from the right side we get
$$
e_{i,i}a_3Ie_{k,k}-e_{i,i}a_3e_{i,k}+e_{i,k}a_3e_{k,k}
=e_{i,i}dIe_{k,k}-e_{i,i}a_{i,k}e_{i,k}+e_{i,k}a_{i,k}e_{k,k},
$$
$$
e_{i,i}a_3Ie_{k,k}+Ie_{i,i}a_3e_{i,k}-Ie_{i,k}a_3e_{k,k}
=e_{i,i}dIe_{k,k}+e_{i,i}a_{i,k}Ie_{i,k}-Ie_{i,k}a_{i,k}e_{k,k},
$$
Hence,
$$
a_3^{i,k}Ie_{i,k}-a_3^{i,i}e_{i,k}+a_3^{k,k}e_{i,k}
=d^{i,k}Ie_{i,k}-a_{i,k}^{i,i}e_{i,k}+a_{i,k}^{k,k}e_{i,k},
$$
$$
a_3^{i,k}Ie_{i,k}+Ia_3^{i,i}e_{i,k}-Ia_3^{k,k}e_{i,k}
=d^{i,k}Ie_{i,k}+a_{i,k}^{i,i}Ie_{i,k}-Ia_{i,k}^{k,k}e_{i,k},
$$
Hence,
$$
a_3^{i,k}=d^{i,k}
$$
and
$$
a_{i,k}^{i,k}=d^{i,k}
$$
by (5.10).
Similarly, multiplying the equalities (5.9) by $e_{k,k}$ from the left side and by $e_{i,i}$ from the right side we get
$$
Ia_3^{k,i}-a_3^{k,k}+a_3^{i,i}
=Id^{k,i}-a_{k,k}^{k,i}+a_{i,k}^{i,i},
$$
$$
Ia_3^{k,i}+Ia_3^{k,k}-Ia_3^{i,i}
=Id^{k,i}+Ia_{i,k}^{k,k}-Ia_{i,k}^{i,i}.
$$
Hence,
$$
a_3^{k,i}=d^{k,i}
$$
and
$$
a_{i,k}^{k,i}=d^{k,i}
$$
by (5.10). Thus we have proved (5.5), (5.6).

Now, let us prove (5.7).
Let $i$ be a natural number, and, let
$$
x=x_o-s_{i-1,i}-s_{i,i+1}-s_{i+1,i+2}.
$$
Then, by additivity of $\nabla$ and the equality (5.2), the following equality holds
$$
a_2x_o-x_oa_2=
$$
$$
=a_{i-1,i}s_{i-1,i}-s_{i-1,i}a_{m,i}+a_{i,i+1}s_{i,i+1}-s_{i,i+1}a_{i,i+1}
$$
$$
+a_{i+1,i+2}s_{i+1,i+2}-s_{i+1,i+2}a_{i+1,i+2}+bx-xb.
$$
Multiplying this equality by $e_{i,i}$ on the left side and by $e_{i+1,i+1}$ on the right side we get the equalities:
$$
a_{i,i+1}^{i,i}-a_{i,i+1}^{i+1,i+1}=a_2^{i,i}-a_2^{i+1,i+1}=d^{i,i}-d^{i+1,i+1}.
$$

Let $i$, $k$ be natural numbers. We suppose that $i<k$.
Then we similarly have
$$
a_{i,i+1}^{i,i}-a_{i,i+1}^{i+1,i+1}=a_2^{i,i}-a_2^{i+1,i+1}=d^{i,i}-d^{i+1,i+1},
$$
$$
a_{i+1,i+2}^{i+1,i+1}-a_{i+1,i+2}^{i+2,i+2}=a_2^{i+1,i+1}-a_2^{i+2,i+2}=d^{i+1,i+1}-d^{i+2,i+2},
$$
$$
a_{i+2,i+3}^{i+2,i+2}-a_{i+2,i+3}^{i+3,i+3}=a_2^{i+2,i+2}-a_2^{i+3,i+3}=d^{i+2,i+2}-d^{i+3,i+3},
$$
$$
\dots
$$
$$
a_{k-2,k-1}^{k-2,k-2}-a_{k-2,k-1}^{k-1,k-1}=a_2^{k-2,k-2}-a_2^{k-1,k-1}=d^{k-2,k-2}-d^{k-1,k-1},
$$
$$
a_{k-1,k}^{k-1,k-1}-a_{k-1,k}^{k,k}
=a_2^{k-1,k-1}-a_2^{k,k}=d^{k-1,k-1}-d^{k,k}.
$$
But, for any natural number $m$ and a projection $e$ in $B(H)_{sk}$
such that $e=\sum_{t=1}^le_{m_t,m_t}$, $ee_{m,m}=e_{m,m}$, where
$\{e_{m_1,m_1},e_{m_2,m_2},\dots,e_{m_l,m_l}\}\subset \{e_{i,i}\}_{i=1}^\infty$, the maps
$$
x\to e_{m,m}\nabla(x)e_{m,m}, x\in e_{m,m}B(H)_{sk}e_{m,m}
$$
$$
x\to e\nabla(x)e, x\in eB(H)_{sk}e
$$
are local derivations and, hence, are also derivations by \cite{AK}. Hence,
there exist elements $v\in e_{m,m}B(H)_{sk}e_{m,m}$, $w\in eB(H)_{sk}e$ such that
$$
e_{m,m}\nabla(x)e_{m,m}=vx-xv, x\in e_{m,m}B(H)_{sk}e_{m,m},
$$
$$
e\nabla(x)e=wx-xw, x\in eB(H)_{sk}e.
$$
Besides, $e_{m,m}we_{m,m}=v$. Therefore,
$$
a_{i,i+1}^{i,i}=a_{i,k}^{i,i}, a_{i,i+1}^{i+1,i+1}=a_{i+1,i+2}^{i+1,i+1},
a_{i+1,i+2}^{i+2,i+2}=a_{i+2,i+3}^{i+2,i+2},
a_{i+2,i+3}^{i+3,i+3}=....
$$
$$
....=a_{k-2,k-1}^{k-2,k-2},
a_{k-2,k-1}^{k-1,k-1}=a_{k-1,k}^{k-1,k-1},
a_{k-1,k}^{k,k}=a_{i,k}^{k,k}
$$
by the first part of the present section and lemma \ref{4.0}.
Hence, we have the equality (5.7):
$$
a_{i,k}^{i,i}-a_{i,k}^{k,k}=a_2^{i,i}-a_2^{k,k}=d^{i,i}-d^{k,k}.
$$
So, the equalities (5.4) are valid.

Now, let $x$ be an arbitrary element in $B_{sk}(H)$ and
$$
e_{i,i}xe_{j,j}+e_{j,j}xe_{i,i}=x^{i,j}_1s_{i,j}+Ix^{i,j}_2\bar{e}_{i,j},
$$
$$
e_{i,i}xe_{i,i}=Ix^{i,i}e_{i,i},
$$
where $x^{i,j}_1$, $x^{i,j}_2$, $x^{i,i}$ are real numbers for every pair of different natural numbers $i$, $j$. Then, by the
equalities (5.3), (5.4), we have
$$
\nabla(x^{i,j}_1s_{i,j})=x^{i,j}_1\nabla(s_{i,j})=x^{i,j}_1R_d(s_{i,j})=R_d(x^{i,j}_1s_{i,j}),
$$
$$
\nabla(Ix^{i,j}_2\bar{e}_{i,j})=x^{i,j}_2\nabla(I\bar{e}_{i,j})=x^{i,j}_2R_d(I\bar{e}_{i,j})=R_d(Ix^{i,j}_2\bar{e}_{i,j}),
$$
$$
\nabla(Ix^{i,i}e_{i,i})=x^{i,i}\nabla(Ie_{i,i})=x^{i,i}R_d(Ie_{i,i})=R_d(Ix^{i,i}e_{i,i}),
$$
since $\nabla$ is linear.

Let $(x_\alpha)$ be the net of all elements of the form
$$
\sum_{k=1,l=2, k<l}^m (x^{i_k,i_l}_1s_{i_k,i_l}+Ix^{i_k,i_l}_2\bar{e}_{i_k,i_l})+\sum_{k=1}^m x^{i_k,i_k}Ie_{i_k,i_k},
$$
where $\alpha=\{i_1,i_2,...,i_m\}$ is an arbitrary finite subset of natural numbers. Then the net
$(x_\alpha)$ ultraweakly converges to $x$. We have
$$
\nabla(\sum_{k=1,l=2, k<l}^m (x^{i_k,i_l}_1s_{i_k,i_l}+Ix^{i_k,i_l}_2\bar{e}_{i_k,i_l})+\sum_{k=1}^m x^{i_k,i_k}Ie_{i_k,i_k})
$$
$$
=\sum_{k=1,l=2, k<l}^m (x^{i_k,i_l}_1\nabla(s_{i_k,i_l})+x^{i_k,i_l}_2\nabla(I\bar{e}_{i_k,i_l}))+\sum_{k=1}^m x^{i_k,i_k}\nabla(Ie_{i_k,i_k})=
$$
$$
\sum_{k=1,l=2, k<l}^m (x^{i_k,i_l}_1R_d(s_{i_k,i_l})+x^{i_k,i_l}_2R_d(I\bar{e}_{i_k,i_l}))+\sum_{k=1}^m x^{i_k,i_k}R_d(Ie_{i_k,i_k})
$$
$$
=R_d(\sum_{k=1,l=2, k<l}^m (x^{i_k,i_l}_1s_{i_k,i_l}+Ix^{i_k,i_l}_2\bar{e}_{i_k,i_l})+\sum_{k=1}^m x^{i_k,i_k}Ie_{i_k,i_k}),
$$
i.e., $\nabla(x_\alpha)=R_d(x_\alpha)$ for every $\alpha$. The Lie multiplication is ultraweakly continuous,
so the net $(R_d(x_\alpha))$ is ultraweakly converges to $R_d(x)$. Since $\nabla$ is ultraweakly continuous we have
$$
\nabla(x)=R_d(x).
$$
This completes the proof.
\end{proof}

\section{Local derivations on Lie algebras of skew-adjoint matrix-valued maps}

In this section we describe local derivations on some class of subalgebras of $F(\Omega,B_{sk}(H))$.

Given $t\in \Omega$, $\phi_t : F(\Omega,B(H))\to B(H)$ will denote the
$*$-homomorphism defined by $\phi_t(x) = x(t)$, $x\in F(\Omega,B(H))$. The space $F(\Omega,B(H))$ is an $B(H)$-bimodule
with products $(ax)(t)=ax(t)$ and $(xa)(t)=x(t)a$, for every $a\in B(H)$,
$x\in F(\Omega,B(H))$. The map $\phi_t : F(\Omega,B(H))\to B(H)$ is an $B(H)$-module homomorphism.

Given an arbitrary set $\Omega$, the $*$-homomorphism which maps each element $a$ in $B(H)$ to the constant function
$\Omega\to \{a\}$ will be denoted by $\hat{a}$. The map
$$
\psi(a)=\hat{a}, a\in B(H),
$$
is a $B(H)$-module homomorphism from $B(H)$ to $F(\Omega,B(H))$.

\begin{definition}
Let $A$ be a Lie algebra and, let $B$ be a Lie subalgebra of $A$.
A local derivation $\nabla$ on $B$ is called local spatial derivation
implemented by elements from $A$, if
for every element $x\in B$ there exists a spatial derivation $D$
on $B$ implemented by an element in $A$ such that $\nabla(x)=D(x)$.
\end{definition}

Let $\mathcal{A}=B_{sk}(H)$ or $\mathcal{A}=\mathcal{K}_{sk}(H)$, and let $\Omega$ be a topological space.
Let $F(\Omega,\mathcal{A})$ be the Lie algebra of all maps from $\Omega$ to $\mathcal{A}$, and let
$C(\Omega,\mathcal{A})$ be the Lie algebra of all continuous maps from $\Omega$ to $\mathcal{A}$.
Then we have the following theorem

\begin{theorem} \label{5.1}
Let $\mathcal{L}$ be one of the Lie algebras $F(\Omega,\mathcal{A})$ or $C(\Omega,\mathcal{A})$
and let $\nabla$ be a local spatial derivation on $\mathcal{L}$ implemented by elements from $F(\Omega,B_{sk}(H))$.
Suppose that, for every $t\in \Omega$, the map
$\phi_t\nabla \psi$ is ultraweakly continuous on $B_{sk}(H)$.
Then $\nabla$ is a spatial derivation.
\end{theorem}

\begin{proof}
Similar to the proof of \cite[Theorem 2.4]{JP} we can prove that the map
$\phi_t\nabla \psi:B_{sk}(H)\to B_{sk}(H)$ is a local inner derivation for every $t\in \Omega$.
Hence by Theorem \ref{4.4} $\phi_t\nabla \psi$ is an inner derivation on $B_{sk}(H)$ since
$\phi_t\nabla \psi$ is ultraweakly continuous.

Let $a_t$ be an element in $B_{sk}(H)$ such that
$$
\phi_t\nabla \psi (x)=a_tx-xa_t, x\in B_{sk}(H).
$$
Let $\hat{a}(t)=a_t$, $t\in \Omega$. Then
$$
\nabla ({\bf x})(t)=(\hat{a}{\bf x}-{\bf x}\hat{a})(t), {\bf x}\in \mathcal{L}, t\in \Omega,
$$
i.e., $\nabla ({\bf x})=\hat{a}{\bf x}-{\bf x}\hat{a}$, ${\bf x}\in \mathcal{L}$ and
$\hat{a}\in F(\Omega,B_{sk}(H))$. Hence, $\nabla$ is a spatial derivation
since $\nabla ({\bf x})\in\mathcal{L}$ for any ${\bf x}\in \mathcal{L}$.
\end{proof}

\begin{remark}
Let $\mathcal{L}$ be one of the Lie algebras $F(\Omega,\mathcal{A})$ or $C(\Omega,\mathcal{A})$.
Note that, it is not necessary that for each pair of elements $x\in \mathcal{L}$ and $y\in F(\Omega,B_{sk}(H))$ the elements $yx$, $xy$, $R_x(y)$ belong to $\mathcal{L}$.
But, in theorem \ref{5.1}, for every $x\in \mathcal{L}$, the element $R_a(x)$ belongs to $\mathcal{L}$.
\end{remark}

\end{document}